\documentclass[12pt]{amsart}

\usepackage{amsfonts, amssymb, amsthm, amsmath, bm, braket, enumitem, hyperref, mathtools, comment, mathrsfs}
\usepackage[all]{xy}
\newtheorem{theorem}{Theorem}[section]

\newtheorem{obs}[theorem]{Observation}
\newtheorem{corollary}[theorem]{Corollary}
\newtheorem{cor}[theorem]{Corollary}
\newtheorem{lemma}[theorem]{Lemma}
\newtheorem{proposition}[theorem]{Proposition}
\newtheorem{prop}[theorem]{Proposition}
\newtheorem{example}[theorem]{Example}

\theoremstyle{definition}
\newtheorem{definition}[theorem]{Definition}
\newtheorem{defn}[theorem]{Definition}

\newtheorem{quest}[theorem]{Question}

\newcounter{casenum}
\newenvironment{caseof}{\setcounter{casenum}{1}}{\vskip.5\baselineskip}
\newcommand{\case}[2]{\vskip.5\baselineskip\par\noindent {\bfseries Case \arabic{casenum}:} \textit{#1}\\#2\addtocounter{casenum}{1}}

\newcommand{\defi}[1]{\textsf{#1}}

\newcommand{\bbA}{\mathbb{A}}
\newcommand{\bbB}{\mathbb{B}}
\newcommand{\bbC}{\mathbb{C}}

\newcommand{\bbI}{\mathbb{I}}

\newcommand{\bbN}{\mathbb{N}}

\newcommand{\bbQ}{\mathbb{Q}}
\newcommand{\bbR}{\mathbb{R}}

\newcommand{\bbT}{\mathbb{T}}

\newcommand{\bbZ}{\mathbb{Z}}

\newcommand{\frakA}{\mathfrak{A}}
\newcommand{\frakS}{\mathfrak{S}}

\newcommand{\prm}{^\prime}
\newcommand{\parent}[1]{\left( #1 \right)}

\newcommand{\inv}{^{-1}}
\DeclareMathOperator{\into}{\hookrightarrow}

\newcommand{\Sum}{\frakS}
\newcommand{\Add}{\frakA}

\newcommand{\Domain}{{\rm D}}
\newcommand{\Codomain}{{\rm E}}

\newcommand{\Sums}[2]{\textbf{\upshape{S}}\parent{#1,#2}}
\newcommand{\WSums}[2]{\textbf{\upshape{wMS}}\parent{#1,#2}}
\newcommand{\MSums}[2]{\textbf{\upshape{MS}}\parent{#1,#2}}
\newcommand{\R}{{\rm R}}

\newcommand{\Reg}[1]{{\rm Reg}\parent{#1}}

\newcommand{\Zeroes}[2]{{\rm Z}\parent{#1; #2}}

\newcommand{\T}{\mathcal{T}}
\newcommand{\M}{\mathcal{M}}
\newcommand{\Q}{\mathcal{Q}}
\newcommand{\U}{\mathcal{U}}

\newcommand{\Alg}[1]{\bbA\parent{#1}} 
\newcommand{\AbsAlg}[1]{\bbB\parent{#1}} 
\newcommand{\Trans}[1]{\bbT\parent{#1}} 
\newcommand{\Inf}[1]{\bbI\parent{#1}} 

\newcommand{\scalpoly}[1]{s_{#1}}
\newcommand{\scalroot}[1]{\rho_{#1}}

\newcommand{\reflected}[1]{{#1}^{\text{R}}}

\begin{document}
\title{Multiplicative summations into algebraically closed fields}
\author{Robert J. MacG. Dawson}\email{rdawson@cs.smu.ca}
\author{Grant Molnar}\email{Grant.S.Molnar.GR@dartmouth.edu}\thanks{The second author received support from the Gridley Fund for Graduate Mathematics.}

\begin{abstract}
In this paper, extending our earlier program, we derive maximal canonical extensions for multiplicative summations into algebraically closed fields. We show that there is a well-defined analogue to minimal polynomials for a series algebraic over a ring of series, the ``scalar polynomial''. When that ring is the domain of a summation $\Sum$, we derive the related concepts of the $\Sum$-minimal polynomial for a series, which is mapped by $\Sum$ to a scalar polynomial. When the scalar polynomial for a series has the form $(t-a)^n$, $a$ is the unique value to which the series can be mapped by an extension of the original summation. 
\end{abstract}

\maketitle

\section{Introduction}\label{Section: Introduction}

This paper is a sequel to \cite{DM2}, and we follow the notations and conventions established therein. Let $\R$ be a commutative unital ring with $0 \neq 1$, and let $\Codomain$ be a commutative unital $\R$-algebra with $0 \neq 1$. Let $\Domain$ be a set and $\Sum : \Domain \to \Codomain$ a map. 

\begin{defn}
	The tuple $(\R, \Codomain, \Domain, \Sum)$ is a \defi{summation on $\R$ to $\Codomain$} if it satisfies the following axioms:
	\begin{enumerate}[label=(\Roman*)]
		\item We have $\R[\sigma] \subseteq \Domain \subseteq \R[[\sigma]]$;\label{Condition: AreSeries}
		\item The set $\Domain$ is an $\R$-module, and the map $\Sum$ is an $\R$-module homomorphism; \label{Condition: IsModuleMorphism}
		\item We have $\Sum(1) = 1$; \label{Condition: ExtendsAdd}
		\item We have $(1 - \sigma) \Domain \subseteq \Domain$, and the morphism $\Sum$ factors through $\Domain / (1 - \sigma) \Domain$. \label{Condition: Factors by (1 - sigma)}
	\end{enumerate}
\end{defn}

\noindent Axiom \ref{Condition: Factors by (1 - sigma)} gives us the following commutative diagram of $\R$-modules:
	
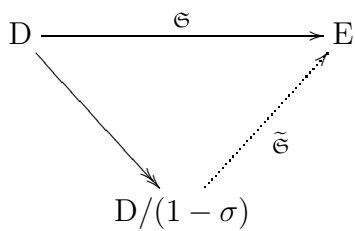
\[
		\xymatrix@R=50pt{
		\Domain \ar[rr]^{\Sum} \ar@{->>}[dr]& & \Codomain \\
		& \Domain /(1 - \sigma) \ar@{.>}[ur]_{\widetilde{\Sum}} &
		}
	\]

We write $(\Domain, \Sum)$ or simply $\Sum$ for the summation $(\R, \Codomain, \Domain, \Sum)$ when no confusion results. We write $\Sums \R \Codomain$ for the set of all summations on $\R$ to $\Codomain$. If $X$ is a power series in the variable $\sigma$, we write $(X)_n$ for the $n$th coefficient of $X$, or $X_n$ if no confusion arises. We typically use capital letters for series over $\R$, and lower-case letters for scalars in $\Codomain$.

Our definition of summations leaves open the possibility that the underlying ring $\R$ and the codomain $\Codomain$ are not the same. Classically, this freedom allows us to sum a rational series into $\bbR$; telescoping also allows us to sum an integral series into $\bbQ$. But we do not even demand that $\R$ is a subring of $\Codomain$. If it is not, then the map $x \mapsto \Sum(x+0+0+\cdots)$ is not injective. We call a summation \defi{proper} if $x \mapsto \Sum(x+0+0+\cdots)$ is injective. The following example illustrates this distinction.

\begin{example}\label{Example: Defining Add}
	We define $\Add : \R[\sigma] \to \R$ by $\Add : X(\sigma) \mapsto X(1)$. This is the unique minimal summation over $\R$ with values in $\R$; it is proper.\\  
        Given a ring homomorphism $f:\R\rightarrow\Codomain$, the composition $f\circ\Add$ is a summation over $\R$ with values in $\Codomain$, which by slight overloading of notation we can call $\Add : \R[\sigma] \to \Codomain$. It is proper if and only if $f$ is injective.\\
        For instance, $\Add:\bbZ[\sigma]\rightarrow \bbZ/2\bbZ$ is a summation that maps $1+2+3$ to $0$, and is not proper.
\end{example}

Recall that the \defi{fulfillment} of a summation $\Sum \in \Sums \R \Codomain$ is its unique maximal canonical extension. We write $\Reg \Codomain$ for the nonzero divisors of $\Codomain$. 

\begin{defn}\label{Definition: Telescopic Extension}
	For any summation $(\Domain, \Sum) \in \Sums \R \Codomain$, the \defi{telescopic extension} $(\T \Domain, \T \Sum)$ of $(\Domain, \Sum)$ is defined as follows. For $X \in \R[[\sigma]]$, we say $X \in \T\Domain$ if there exists $A \in \Domain$, $F \in \R[\sigma]$, $f \in \Reg \Codomain$, and $x \in \Codomain$ such that $A = FX$, $\Sum(F) = f$, and $\Sum(A) = f x$. We define
	\begin{align*}
		\T \Sum &: \T \Domain \to \Codomain, \\
		\T \Sum &: X \mapsto x \ \text{if} \ X \ \text{is as above.}
	\end{align*}
\end{defn}

\begin{theorem}\label{Theorem: Telescopic Extension}
	The functor of summations 
	\[
		\T : \Sums \R \Codomain \to \Sums \R \Codomain
	\]
	is an idempotent extension map. Moreover, if $\R$ is an integral domain, then $\T \Sum$ is the fulfillment of $\Sum$.
\end{theorem}

The tuple $(\R, \Codomain, \Domain, \Sum)$ is a \defi{multiplicative summation on $\R$ to $\Codomain$} if it satisfies axioms \ref{Condition: AreSeries}, \ref{Condition: ExtendsAdd}, \ref{Condition: Factors by (1 - sigma)} from Section \ref{Section: Introduction}, and the following strengthened version of axiom \ref{Condition: IsModuleMorphism}:
\begin{enumerate}[label=(\Roman*$\prm$)]
	\setcounter{enumi}{1}
	\item The set $\Domain$ is an $\R$-algebra, and the map $\Sum$ is an $\R$-algebra homomorphism; \label{Condition: IsAlgebraMorphism}
\end{enumerate}

\noindent In this context, axiom \ref{Condition: Factors by (1 - sigma)} gives us the following commutative diagram of $\R$-algebras:
	\[
		\xymatrix@R=50pt{
		\Domain \ar[rr]^{\Sum} \ar@{->>}[dr]& & \Codomain \\
		& \Domain /(1 - \sigma) \ar@{.>}[ur]_{\widetilde{\Sum}} &
		}
	\]

\begin{definition}
	If a summation has an extension which is multiplicative, we call it \defi{weakly multiplicative}.
\end{definition}

We considered this property in detail previously \cite{DM2}. In particular, we show that it is possible for a summation to preserve all products that exist in its domain, and yet not be weakly multiplicative \cite{DM2}[Example 3.7]. However, if a summation $\Sum$ is weakly multiplicative, then it has a unique minimal multiplicative extension $\M \Sum$. We write $\WSums \R \Codomain$ for the set of weakly multiplicative summations on $\R$ to $\Codomain$; organized by inclusion, $\WSums \R \Codomain$ is the full subcategory of $\Sums \R \Codomain$ with objects the weakly multiplicative summations. 

\begin{defn}\label{Definition: Rational Extension}
	For a multiplicative summation $(\Domain, \Sum) \in \MSums \R \Codomain$, the \defi{rational extension} $(\Q\Domain, \Q\Sum)$ of $(\Domain, \Sum)$ is defined as follows. For $X \in \R[[\sigma]]$, we say $X \in \Q\Domain$ if there exists $A, B \in \Domain,$ $b \in \Reg \Codomain$, and $x \in \Codomain$ such that $A = BX$, $\Sum(B) = b$, and $\Sum(A) = b x$. We define
	\begin{align*}
		\Q\Sum &: \Q\Domain \to \Codomain, \\
		\Q\Sum &: X \mapsto x \ \text{if} \ X \ \text{is as above.}
	\end{align*}
	We extend this definition to weakly multiplicative summations $(\Domain, \Sum) \in \WSums \R \Codomain$ by setting $(\Q\Domain, \Q\Sum) \coloneqq (\Q\M\Domain, \Q\M\Sum)$.
\end{defn}

Repeating (essentially) the proof of Theorem \ref{Theorem: Telescopic Extension}, we obtain the following theorem \cite{DM2}[Theorem 5.3].

\begin{theorem}\label{Theorem: Rational Extension}
	The functor of summations
	\[
		\Q : \WSums \R \Codomain \to \MSums \R \Codomain
	\]
	is an idempotent extension map which subsumes $\M$, $\T$, $\M \T$, and $\T \M$.
\end{theorem}

If $\Sum = \Q\Sum$, we say that $\Sum$ is \defi{rationally closed}.

\begin{prop}\label{Proposition: The image of Q(Sum) is the field of fractions of Sum(Domain)}
	Let $\Codomain$ be a field. For any multiplicative summation $(\Domain, \Sum) \in \MSums \R \Codomain$, the image of $\Q \Domain$ under $\Q\Sum$ is the field of fractions of $\Sum(\Domain)$.
\end{prop}

In \cite{DM2} we posed the following question.	

\begin{quest}
	What is the multiplicative fulfillment of a weakly multiplicative summation?
\end{quest}

In this paper, we answer this question for multiplicative summations whose codomains are algebraically closed fields (see Theorems \ref{Theorem: Univalent Extension} and \ref{Theorem: Univalent Extension is multiplicative fulfillment}).

The remainder of this paper is organized as follows. Fix a summation $\Sum$. In Section \ref{Section: The Algebra of Summations} we define the scalar polynomial of a series $X \in \R[[\sigma]]$, and and show that the roots of this scalar polynomial determine all images of $X$ under extensions of $\Sum$. In Section \ref{Section: Absolutely Algebraic Series}, we make a more careful study of the absolutely $\Sum$-algebraic series, which are precisely those series contained in every multiplicative extension of $\Sum$ (see Proposition \ref{Proposition: The absolutely algebraic series are the intersection of all superdomains}). In Section \ref{Section: Univalent Extensions}, we use the theory of absolutely $\Sum$-algebraic series to characterize the multiplicative fulfillment of $\Sum$. Finally, in Section \ref{Section: Future work} we discuss some directions for future work.

\section*{Acknowledgments}

We thank Asher Auel for insightful conversations.

\section{The Algebra of Summations}\label{Section: The Algebra of Summations}

Throughout the remainder of this paper, we assume our summations are multiplicative, and when we speak of extensions, fulfillments, and so forth, we mean \emph{multiplicative} extensions, \emph{multiplicative} fulfillments, and so forth. We also assume that the codomains of our (multiplicative) summations are \emph{algebraic closed fields} unless otherwise noted. This is not an onerous assumption; if $\Codomain$ is a field and $\overline{\Codomain}$ is the algebraic closure of $\Codomain$, the inclusion map $\iota : \Codomain \into \overline{\Codomain}$ induces injections 
\begin{align*}
\MSums \R \Codomain &\into \MSums \R {\overline{\Codomain}}, \\ 
\WSums \R \Codomain &\into \WSums \R {\overline{\Codomain}}, \\ 
\Sums \R \Codomain &\into \Sums \R {\overline{\Codomain}}
\end{align*}
via $\Sum \mapsto \iota \circ \Sum$.

A \defi{series polynomial} is simply a polynomial $P(T) = \sum\limits_{k = 0}^n P_k T^k$ with coefficients in $\R[[\sigma]]$. If $P(T) \in \Domain(T)$, we write $\Sum(P)(t)$ for the polynomial $\sum\limits_{k = 0}^n \Sum (P_k) t^k \in \Codomain[t]$. 

\begin{definition}
	For a series $X \in \R[[\sigma]]$ and a summation $(\R,\Codomain,\Domain,\Sum)$ , let 
\[
{\rm A}(X; \Sum) \coloneqq \set{P(T) \ : \ P(T) \in \Domain[T], \ \Sum(P)(t) \neq 0, \ P(X)=0}.
\]

A polynomial $P(T)$ in ${\rm A}(X; \Sum)$ for which $\Sum(P)(t)$ has minimal degree will be called a \defi{$\Sum$-minimal polynomial} for $X$. If ${\rm A}(X; \Sum)$ is nonempty, we say that the \defi{$\Sum$-degree} of $X$ is the smallest degree of a $\Sum$-minimal polynomial for $X$.
\end{definition}

\begin{example}  The series 
\[
X = 2 - \frac{\sigma}4  - \frac{\sigma^2}{64}-\frac{\sigma^3}{512}- \frac{5\sigma^4}{16384}- \cdots,
\]
obtained as the Taylor series of $\sqrt{4-\sigma}$, satisfies $X^2 - (4-\sigma) = 0$, but satisfies no nontrivial linear equation with coefficients in $\bbC[\sigma]$. Thus the polynomial $P(T) = T^2 - (4-\sigma)$ is minimal and $\Add$-minimal for $X$ over $\bbC[\sigma]$, and $\Add(P)(t)=t^2-3$.
\end{example} 

\begin{example}\label{AddInfinite} The series 
\[
Y = 1+\frac{\sigma}{2} + \frac{3\sigma^2}{8} + \frac{5\sigma^3}{16} + \frac{35\sigma^4}{128}+ \cdots,
\]
obtained as the Taylor series of $(1-\sigma)^{-1/2}$, satisfies $(1-\sigma)Y^2 - 1 = 0$, but satisfies no nontrivial linear equation with coefficients in $\bbC[\sigma]$. The polynomial $Q(T) = (1-\sigma)T^2 - 1$ is minimal and $\Add$-minimal for $Y$ over $\bbC[\sigma]$, and $\Add(Q)(t) = -1$ is constant.  
\end{example}

\begin{prop}\label{Proposition: proper sums don't let calM have monomials}
	If $\R$ is an integral domain, and $X\in\R[[\sigma]]$ is nonzero, then no monomial is an element of ${\rm A}(X; \Sum)$.
\end{prop}

\begin{proof}
Let $A_j$ and $X_k$ be the first nonzero coefficients of $A$ and $X$ respectively. We see $(AX^n)_{j+kn}\neq 0$.
\end{proof}

Example \ref{dual} below shows Proposition \ref{Proposition: proper sums don't let calM have monomials} need not hold if $\R$ has zero divisors.

Given a summation $(\R,\Codomain,\Domain,\Sum)$ and a series $X \in \R[[\sigma]]$, it is possible that a polynomial $P(T) \in \Domain[T]$ satisfies $P(X)=0$ for some series $X\in\R[[\sigma]]$, but that $\Sum(P)(t) = 0$. In this case, the polynomial relation $P(X) = 0$ implies nothing about the possible values of $X$ in an extension of $\Sum$. This circumstance may even be universal among polynomial $P(T) \in \Domain[T]$ with $P(X) = 0$, as Example \ref{PseudoTrans} shows. Such an $X$ may be formally algebraic over $\Domain$, but still ``behaves like a transcendental series'' in that $X$ may consistently be assigned any value in $\Codomain$ in an appropriate extension of $\Sum$ (see Theorem \ref{Theorem: Summation-structure of Sum-*** Series}).  

\begin{example}\label{PseudoTrans} Let $U,V \in \bbC[[\sigma]]$ be invertible (\emph{i.e.} $U_0$ and $V_0$ are nonzero) and algebraically independent over $\bbC[\sigma]$, and let $\Domain \coloneqq \bbC[\sigma][U,V]$. We see that $\Domain$ is the  ring of all series of the form $\sum_{i, j} A_{ij}U^iV^j$ where $A_{ij}\in\bbC[\sigma]$ and the double sum is finitely supported. Let $\Sum(U)=\Sum(V)=0$; then $\Sum \left(\sum A_{ij}U^iV^j\right ) = \Add(A_{00})$. If we let $X \coloneqq UV^{-1}$, then $VX-U=0$, and $P(X) = 0$ implies $\Sum(P)(t) = 0$.
\end{example}

However, the next proposition says that this cannot happen for $(\R[\sigma],\Add)$ if addition is proper (that is, if $x\mapsto\Add(x+0+0+\cdots)$ is injective).

\begin{prop}\label{NotNull} For any ring $\R$, if $\Add$ is proper, $X \in \R[[\sigma]]$, $P \in \R[\sigma][T]$, $P(X) = 0$, and $P \neq 0$,  then there exists $P' \in \R[\sigma][T]$ such that $\Add(P')(t) \neq 0$ and $P'(X)=0$.
\end{prop}

\begin{proof}
If $x \mapsto \Add(x+0+0+\cdots)$ is injective, then for $P \in \R[\sigma]$, we have $\Add(P)(t) = 0$ if and only if $1-\sigma$ divides $P$.

Suppose now that $\Add$ is proper, $P(X) = 0$, and $P \neq 0$. If $\Add(P)(t) \neq 0$, then we are done. Otherwise, by properness, the term $1 - \sigma$ divides $P(T)$. By induction the power of $1 - \sigma$ dividing $P(T)$, there exists $n > 0$ and $P\prm(T) \in \R[\sigma][T]$ such that 
\[
P(T) = (1-\sigma)^n P\prm(T) = (1 - \sigma)^n \sum_k P'_k T^k
\]
and $\Add(P\prm)(t) \neq 0$. By assumption,
\[
0 = P(X) = (1-\sigma)^n \sum_k P'_k X^k,
\]
and $\Add(P'_k)\neq 0$ for some $k$. But multiplication by $(1-\sigma)$ leaves the first nonzero term of a series unchanged; so $\sum_k P'_k X^k = 0$ as claimed.
\end{proof}

\begin{example}\label{Grandi} The Grandi series 
\[
G_{-1} \coloneqq \sum_{n=0}^\infty (-1)^n \sigma^n = 1-\sigma+\sigma^2-\sigma^3+\cdots = 1-1+1-1+\cdots
\]
can (infelicitously) be telescoped with shift 2, via the equation $(1-\sigma^2)G_{-1} - (1-\sigma) = 0$. Summing the coefficients of the polynomial $(1 - \sigma^2)T - (1 - \sigma)$ yields $0t-0$, which places no constraints on the sum of the Grandi series. However, dividing our polynomial by $1-\sigma$, we obtain $(1+\sigma)T - 1$ which sums to $2t-1$, so the Grandi series should sum to $\frac{1}{2}$ (as it does using many classical summation methods).
\end{example}

The next example illustrates why we stipulate that $\Add$ is proper in Proposition \ref{NotNull}.

\begin{example}\label{NotInjective} Let $\Add_2:\bbZ[\sigma]\rightarrow\bbZ/2\bbZ$, and let $X \coloneqq 2G_{-1} = 2-2+2-2+\cdots$; this has minimal polynomial $P_X(T) \coloneqq (1+\sigma)T-2$. But $\Add_2(P_X)(t) = 0t+0$, the null polynomial, and so $\Add_2(P_X)(t) = 0$ for all $t\in\bbZ/2\bbZ$.
\end{example}

In general, $\Sum$-minimal polynomials need not be minimal, and minimal polynomials need not be $\Sum$-minimal. 

\begin{example}\label{dual}
	Let $\R = \bbC[\epsilon]/\epsilon^2\bbC[\epsilon]$ be the dual numbers over $\bbC$, let $\Codomain = \bbC$, and let $\Sum \in \MSums {\R}{\bbC}$ be the summation with domain 
\[
	\Domain \coloneqq \set{X + \epsilon Y \ : \ X, \ Y \in \R[\sigma]},
\] which is defined by $\Sum : X + \epsilon Y \mapsto \Add(X)$. Finally, let $Y \coloneqq \frac{\epsilon}{1 - \sigma} = \epsilon + \epsilon + \epsilon + \ldots$, let $F(T) = \epsilon T$, and let $G(T) = T^2 + \epsilon H(T)$ where $H(T) \in \R[T]$ is arbitrary. Then $F(T)$ is a minimal polynomial for $Y$ over $\Domain$, but as $\Sum (F) (t) = 0$, $F(T)$ is not a $\Sum$-minimal polynomial for $Y$. On the other hand, $G(T)$ is a $\Sum$-minimal polynomial for $Y$, but not a minimal polynomial. By varying $H(T)$, we also see $\Sum$-minimal polynomials need not be unique even up to scalars or degree.
\end{example}

In the last example, the polynomial $F(T)$ failed to be $\Sum$-minimal because $\Sum$ mapped it to the null polynomial, which is treated specially in the definition. (Proposition \ref{NotNull} above shows that this cannot happen for $\Add$.)  A minimal polynomial can also fail to be $\Sum$-minimal because $\Sum$ lets another polynomial ``jump the queue''.

\begin{example}
Let 
$$X \coloneqq 1 + \frac{\sigma}{4} + \frac{9\sigma^2}{64} + \frac{49 \sigma^3}{512} + \frac{1165\sigma^4}{16384} + \cdots$$ 
be the unique real series satisfying $(1-\sigma)X^3 + X - 2 = 0$, and let
$$A \coloneqq X^2 = 1 + \frac{\sigma}{2} + \frac{11\sigma^2}{32} + \frac{67 \sigma^3}{256} + \frac{1719\sigma^4}{8192} + \cdots.$$ 
Let $\Domain \coloneqq \bbR[\sigma][A]$ and let $\Sum$ take finitely supported sums $\sum_{j} P_jA^j$ to $\sum_{j} 4^jP_j(1)$. Then $T^2-A$ is minimal for $X$ over $\Domain$, because no linear polynomial with coefficients in $\Domain$ annihilates $X$. However, $(1-\sigma)T^3 + T - 2 = 0$ is $\Sum$-minimal, because its image is the linear polynomial $t-2$, and any lower-degree polynomial annihilating $X$ would be constant, hence necessarily the null polynomial. 
\end{example}

Although $\Sum$-minimal polynomials are not unique, their images under $\Sum$ are, up to a constant: if $P_1(T)$ and $P_2(T)$ are $\Sum$-minimal polynomials for $X$, then $\Sum (P_1) (t) = a \cdot \Sum (P_2) (t)$ for some unit $a \in \Codomain^\times$. The roots of such images are the possible sums for $X$ in an extension of $\Sum$ (see Theorem \ref{Theorem: Summation-structure of Sum-*** Series}). For convenience, we pick a canonical representative.

\begin{definition}
	If $\Sum (P) (t) = 0$ for every $\Sum$-minimal polynomial for $X$, we define the \defi{scalar polynomial} $\scalpoly X (t)$ for $X$ over $\Sum$ to be $\scalpoly X (t) \coloneqq 0$. (Note that this may be true either vacuously, when $X$ is transcendental over $\Domain$, or nonvacuously, as in Example \ref{PseudoTrans}.) Otherwise, we define the \defi{scalar polynomial} $\scalpoly X(t)$ for $X$ over $\Sum$ to be the unique monic polynomial in $\Codomain[t]$ which is a nonzero scalar multiple of $\Sum (P) (t)$ for some (equivalently, for every) $P(T)$ that is $\Sum$-minimal for $X$. We define the \defi{scalar degree} of $X$ to be the degree of its scalar polynomial.
\end{definition}

Note that even if $\R$ is a field, there may be series $X \in \R[[\sigma]]$ for which the degree of $\scalpoly X (t)$ is less than the degree of every $\Sum$-minimal polynomial $P(T)$ for $X$.

\begin{example}
	Let $\R = \Codomain = \bbC$ and let $Y$ be the Taylor series of
$$
		 \frac{1 + \sqrt{1 - 4 \sigma + 4 \sigma^3}}{2 - 2 \sigma}; \label{Not absolutely Sum-algebraic}
$$
	that is,
$$
	Y = 1 - \sigma - 2\sigma^2 - 5\sigma^3 - 13\sigma^4 - 36\sigma^5 - 104\sigma^6 - \dots.
$$
Then $(\sigma - 1) T^2 + T - (\sigma + \sigma^2) \in \bbC[\sigma][T]$ is an $\Add$-minimal polynomial for $Y$, and thus $\scalpoly Y (t) = t - 2$ is the scalar polynomial for $Y$. But $Y$ is not a rational function of $\sigma$, so does not have a $\Add$-minimal polynomial of degree 1. 
\end{example}

\begin{defn}\label{Definition: Sum-*** series}
	We partition the series in $\R[[\sigma]]$ based on their scalar polynomials as follows:
	\begin{itemize}
		\item A series $X \in \R[[\sigma]]$ is \defi{multiplicatively $\Sum$-transcendental} if its scalar polynomial $\scalpoly X (t)$ is 0. We write $\Trans \Sum \subseteq \R[[\sigma]]$ for the set of $\Sum$-transcendental series;
		\item A series $X \in \R[[\sigma]]$ is \defi{multiplicatively $\Sum$-algebraic} if its scalar polynomial $\scalpoly X (t)$ is nonconstant. We write $\Alg \Sum \subseteq \R[[\sigma]]$ for the set of $\Sum$-algebraic series;
		\item A series $X \in \R[[\sigma]]$ is \defi{multiplicatively $\Sum$-infinite} if its scalar polynomial $\scalpoly X (t)$ is 1. We write $\Inf \Sum \subseteq \R[[\sigma]]$ for the set of $\Sum$-infinite series.
	\end{itemize}
\end{defn}

Continuing our standard practice, we suppress the adverb ``multiplicatively'' when discussing $\Sum$-transcendental, $\Sum$-algebraic, and $\Sum$-infinite series where this does not result in ambiguity.

We observe that every summation $\Sum \in \MSums \R \Codomain$, the set of formal series $\R[[\sigma]]$ may be partitioned as 
\begin{align}
	\R[[\sigma]] = \Trans \Sum \sqcup \Alg \Sum \sqcup \Inf \Sum.\label{Partition of R[[sigma]]}
\end{align}
It is easy to see that each transcendental series is $\Sum$-transcendental. On the other hand, each $\Sum$-algebraic series and each $\Sum$-infinite series is algebraic. We note (see Example \ref{AddInfinite}) that multiplicatively $\Sum$-infinite series need not be $\Sum$-infinite under the definition given in \cite{Dawson}, which pertains only to linear equations $AX+B=0$ with $A,B\in\Domain$, $\Sum(A)=0$.

\begin{proposition}\label{Proposition: If P(X) = 0 then the scalar polynomial divides Sum P(X)}
	Let $X \in \R[[\sigma]]$ be any series, and let $Q(T) \in \Domain[T]$ be a polynomial such that $Q(X) = 0$. Then $\scalpoly X(t)$ divides $\Sum (Q) (t)$.
\end{proposition}

\begin{proof}	
	Let $P(t) \in \Domain[t]$ be a $\Sum$-minimal polynomial for $X$, and let $Q(t) \in \Domain[t]$ be any polynomial with $Q(X) = 0$. The proposition is immediate by applying the Euclidean algorithm to $\Sum(P)(t)$ and $\Sum (Q) (t)$.
\end{proof}

\begin{corollary}\label{Corollary: Scalar polynomials lose factors under extensions}
	Suppose $(\Domain\prm, \Sum\prm)$ extends $(\Domain, \Sum)$, and let $X \in \R[[\sigma]]$ be any series. If $\scalpoly X(t)$ is the scalar polynomial for $X$ over $\Domain$, and $\scalpoly X\prm(t)$ is the scalar polynomial for $X$ over $\Domain\prm$, then $\scalpoly X\prm(t)$ divides $\scalpoly X(t)$.
\end{corollary}

Let $P(T) \in \Domain[T]$ be a polynomial of degree $m$. Following \cite{Graham-Knuth-Patashnik}[Chapter 7.3], if $m \neq \infty$, we define the \defi{reflected polynomial} of $P(T)$ to be 
\begin{align}
\reflected{P}(T) &\coloneqq T^m P(T\inv).\label{Definition of reflected polynomial}
\end{align} 
Otherwise $P(T) = 0$, and we define $\reflected{P}(T) \coloneqq 0$.

\begin{corollary}\label{Corollary: scalar polynomial of inverse divides reflected polynomial}
	Let $U \in \R[[\sigma]]$ be a unit, and suppose $P(T) \in \Domain[T]$ is a $\Sum$-minimal polynomial for $U$. Then $\scalpoly {U\inv} (T)$ divides $\Sum (\reflected{P})(T)$ .
\end{corollary}

Write 
\begin{align}
	\Zeroes X \Sum &\coloneqq \set{x \in \Codomain \ : \ \scalpoly X(x) = 0} \label{Definition of Zeros X Sum}
\end{align} 
for the set of zeroes of the scalar polynomial $\scalpoly X (t)$ of $X$. Note that this is also the set of common zeroes of the sums of the polynomials in ${\rm A}(X; \Sum)$.

\begin{cor}
	For every summation $\Sum\prm \in \MSums \R \Codomain$ extending $\Sum$ and every $X \in \R[[\sigma]]$, we have 
	\[
		\Zeroes X \Sum \supseteq \Zeroes X {\Sum\prm}.
	\]
\end{cor}

For $X \in \R[[\sigma]]$ and $x \in \Codomain$, write $\Sum_{X,x} : \Domain[X]\to \Codomain$ for the summation given by
	\[
		\Sum_{X,x} : P(X) \mapsto \Sum (P) (x) \ \text{if} \ P(T) \in \Domain[T],
	\]
	whenever this summation is well-defined. Clearly $\Sum_{X,x}$ is a multiplicative summation whenever it is well-defined.

\begin{theorem}\label{Theorem: Summation-structure of Sum-*** Series}
	Fix a multiplicative summation $(\Domain, \Sum) \in \MSums \R \Codomain$.
	\begin{enumerate}[label=(\alph*)]
		\item Let $X \in \Trans \Sum$. Then for every $x \in \Codomain$, there exists a multiplicative summation $\Sum\prm$ extending $\Sum$ such that $\Sum\prm(X) = x$; moreover, $\Sum_{X,x} : \Domain[X] \to \Codomain$ is the unique minimal multiplicative extension of $\Sum$ which maps $X$ to $x$. \label{Summing Sum-transcendental series}
		\item Let $X \in \Alg \Sum$. Then for every $x \in \Codomain$ a root of $\scalpoly X (t)$, there exists a multiplicative summation $\Sum\prm$ extending $\Sum$ such that $\Sum\prm(X) = x$; moreover, $\Sum_{X,x} : \Domain[X] \to \Codomain$ is the unique minimal multiplicative extension of $\Sum$ which maps $X$ to $x$. Conversely, if $\Sum\prm$ is a multiplicative extension of $\Sum$ for which $\Sum\prm(X)$ is defined, then $\Sum\prm(X)$ is a root of $\scalpoly X (t)$. \label{Summing Sum-algebraic series}
		\item Let $X \in \Inf \Sum$. Then $X$ is not in the domain of any multiplicative extension of $\Sum$. \label{Summing Sum-infinite series}
	\end{enumerate}
\end{theorem}

\begin{proof} 
	\begin{enumerate}[label=(\alph*)]
		\item\label{Extension by arbitrary transcendental series} Let $X$ be $\Sum$-transcendental, and let $x \in \Codomain$ be arbitrary. We claim 
		\[
		\Sum_{X,x} : \Domain[X] \to \Codomain
		\]
		is well-defined; indeed, if $A(X) = B(X)$ then $\Sum (A-B)(t) = 0$ by the definition of $\Sum$-transcendence, and in particular $\Sum (A-B)(x) = 0.$ It is clear that $\Sum_{X,x}$ is a multiplicative summation, and moreover that $\Sum_{X,x}$ is minimal among summations $\Sum\prm$ extending $\Sum$ for which $\Sum\prm (X) = x$.

		\item\label{Extension by roots of scalar polynomials} Let $X$ be $\Sum$-algebraic, and let $x$ be a root of $\scalpoly X (t)$ in $\Codomain$. The proof that $\Sum_{X,x}$ is well-defined is similar to the one given in \ref{Extension by arbitrary transcendental series} above. Suppose that $A(X) = B(X)$ for $A(T), B(T) \in \Domain[T]$. Then $\Sum (A-B)(x)= 0$, and so $\scalpoly X(t)$ divides $\Sum (A-B)(t)$; say $\Sum (A-B)(t) = \scalpoly X(t) \cdot r(t) \in \Codomain[t]$. Then
		\[
			\Sum (A-B)(x) = \scalpoly X(x) \cdot r(x) = 0,
		\]
		and so $\Sum_{X,x}$ is well-defined. 
		
		Conversely, suppose $\Sum\prm$ is a multiplicative extension of $\Sum$ for which $\Sum\prm(X)$ is defined. Choose $P(T) \in \Domain[T]$ a $\Sum$-minimal polynomial, and write $\Sum (P)(t) = a \cdot \scalpoly X(t)$ with $a \in \Codomain^\times$. Then
		\[
			a \cdot \scalpoly X(\Sum\prm(X)) = \Sum\prm(X) = \Sum (P)(X) = 0,
		\] 
		and so $\Sum\prm(X)$ is a root of $\scalpoly X(t)$ as desired.

		\item Let $X$ be $\Sum$-infinite, and choose a polynomial $P(T) \in \Domain[T]$ with $P(X) = 0$ and $\Sum (P)(t) = p \in \Codomain^\times$. Suppose by way of contradiction that $\Sum\prm$ is an extension of $\Sum$ with $X$ in its domain. Then 
		\[
			0 \neq p = \parent{\Sum\prm P}\parent{\Sum\prm X} = \Sum\prm \parent{P(X)} = 0,
		\]
		which is a contradiction.
	\end{enumerate}
\end{proof}

\begin{corollary}\label{Corollary: Alternate Characterization for Sum-*** Series}
	If $\Sum \in \MSums \R \Codomain$ is a multiplicative summation, for any $X \in \R[[\sigma]]$ we have
\begin{equation}
		\Zeroes X \Sum = \set{x \in \Codomain \ : \ x = \Sum\prm (X) \ \text{for some} \ \Sum\prm \supseteq \Sum},  \label{Characterizing Set for Sum-*** Series}
\end{equation}
	and:
	\begin{enumerate}[label=(\alph*)]
		\item the series $X$ is $\Sum$-transcendental if and only if $\Zeroes X \Sum = \Codomain$\label{infinite zeros}
		\item the series $X$ is $\Sum$-algebraic if and only if $\Zeroes X \Sum$ is finite and nonempty;\label{finite zeros}
		\item and the series $X$ is $\Sum$-infinite if and only if $\Zeroes X \Sum = \emptyset$.\label{no zeros}
	\end{enumerate}
\end{corollary}

\begin{obs} As stated above, we are implicitly assuming $\Codomain$ to be algebraically closed. Theorem \ref{Theorem: Summation-structure of Sum-*** Series} does not require this hypothesis; but without it, the corollary is weaker. While \ref{infinite zeros} holds unchanged,  ``if and only if'' is replaced by ``if'' in \ref{finite zeros}, and by ``only if'' in \ref{no zeros}.  
\end{obs}

\begin{example} Let $X$ and $(\Domain,\Sum)$ be as in Example \ref{PseudoTrans}. As shown, $X$ is algebraic over $\Domain$, with minimal polynomial $VX-U$, so $X$ is algebraic; but $\scalpoly X (t) = 0$, so $\Zeroes X \Sum = \bbC$. 
\end{example}

\begin{example} Let $X$ and $(\Domain,\Sum)$ be as in Example \ref{NotInjective}. As shown, $X$ is algebraic over $\Domain$, with minimal polynomial $(1+\sigma)X-2$, so $X$ is algebraic; but $\scalpoly X (t) = 0$, so $\Zeroes X \Sum = \bbZ/2\bbZ$. 
\end{example}

\begin{example}
Consider the summation $(\bbQ[\sigma],\Add_{\bbQ})\in\MSums\bbQ\bbQ$. The Taylor series $X$ for $\sqrt{1+4\sigma}$ is
\[X = 1+ 2 \sigma - 2 \sigma^2 + 4 \sigma^3 - 10 \sigma^4 + 28 \sigma^5 - \cdots.\]
The series $X$ satisfies $X^2 - (1+\sigma) = 0$, and is thus $\Add_{\bbQ}$-algebraic; but $\Zeroes X {\Add_{\bbQ}}$ is empty. 
\end{example}

In \cite{DM2}, we showed that two sums can be consistent without being multiplicatively compatible. In the same spirit we now construct a series $Y$ which is summed to a unique value by every summation compatible with $\Add$, but which nonetheless is not in the multiplicative fulfillment of $\Add$.

\begin{example}\label{PickOne}
	Let $Y$ be as in Example \ref{Not absolutely Sum-algebraic}, and define $Y\prm \coloneqq \frac{1}{1-\sigma} - Y$. We see that $Y$ and $Y\prm$ have the same scalar polynomial $\scalpoly Y (t) = \scalpoly {Y\prm} (t) = t - 2$; therefore, by Theorem \ref{Theorem: Summation-structure of Sum-*** Series}, every summation compatible with $\Add$ that sums $Y$ sums it to $2$, and the same holds for $Y\prm$. Moreover, the summations $\Add_{Y, 2}$ and $\Add_{Y\prm, 2}$ are examples of such summations, so these assertions are nonvacuous. But now if $Y$ and $Y\prm$ were in the multiplicative fulfillment of $\Add$, then 
	\[
	Y + Y\prm = \frac{1}{1-\sigma} =  \sum_{j=0}^\infty \sigma^j
	\]
	would also be in the multiplicative fulfillment of $\Add$, which is absurd.
\end{example}

Example \ref{PickOne} motivates the following definition.

\begin{defn}\label{Definition: Absolutely Sum-algebraic}
	A $\Sum$-algebraic series $X \in \Alg \Sum$ is \defi{absolutely $\Sum$-algebraic} if it is $\Sum\prm$-algebraic for every $\Sum\prm$ extending $\Sum$. We write $\AbsAlg \Sum$ for the set of absolutely $\Sum$-algebraic series; obviously, $\AbsAlg \Sum \subseteq \Alg \Sum$. 
\end{defn}

\begin{example} The series $Y'$ of Example \ref{PickOne} is $\Add$-algebraic, but $\Add_{Y,2}$-infinite; therefore it is not absolutely $\Add$-algebraic. However, as shown in \cite{Dawson}, any series that is telescopable over $\Add$ is telescopable to the same value over any summation whatsoever; thus, for instance, over any field in which $0 \neq 2$, the Grandi series $G_{-1}$ (Example \ref{Grandi}) is absolutely $\Add$-algebraic.
\end{example}

We could define absolutely $\Sum$-transcendental series and absolutely $\Sum$-infinite series similarly, but these definitions would be superfluous; no $\Sum$-transcendental series is absolutely $\Sum$-transcendental, and every $\Sum$-infinite series is absolutely $\Sum$-infinite. 

\begin{prop}\label{Proposition: Every Sum-*** Series is a Sum'-*** Series}
	Let $(\Domain\prm, \Sum\prm)$ be a multiplicative extension of $(\Domain, \Sum)$. Then \begin{enumerate}[label=(\alph*)]
	\item $\Trans \Sum \supseteq \Trans {\Sum\prm}$,
	\item $\AbsAlg \Sum \subseteq \AbsAlg {\Sum\prm}$,
	\item $\Inf \Sum \subseteq \Inf {\Sum\prm}$.
	\end{enumerate}
\end{prop}

\begin{proof}
	Suppose $X \in \Trans {\Sum\prm}$. Then for every polynomial $P(T) \in \Domain\prm[T]$ with $P(X) = 0$, we have $\Sum\prm (P)(t) = 0$. \textit{A fortiori}, for every polynomial $P(T) \in \Domain[T]$ with $P(X) = 0$, we have $\Sum (P)(t) = 0$. Thus $\Trans \Sum \supseteq \Trans {\Sum\prm}$.

	Suppose $X \in \AbsAlg \Sum$. Then $X$ is $\Sum^{\prime \prime}$-algebraic for every extension $\Sum^{\prime \prime}$ of $\Sum$. But every extension of $\Sum\prm$ is also an extension of $\Sum$, and so we see $X \in \AbsAlg {\Sum\prm}$ as desired.

	Finally, suppose $X \in \Inf \Sum$. Then there exists a polynomial $P(T) \in \Domain[T]$ with $P(X) = 0$, and $\Sum (P)(t) = a \in \Codomain^\times$. Then viewing $P(T)$ as an element of $\Domain\prm[T]$, we see that $X \in \Inf {\Sum\prm}$.
\end{proof}

In classical analysis, the convergence or divergence of a series $X$ is unaffected by changing finitely many terms of $X$. Within our axiomatic paradigm for summing series, something analogous holds true.

\begin{defn}\label{Definition: Tail-equivalence}
	Two series $X$ and $Y$ in $\R[[\sigma]]$ are \defi{tail-equivalent} if there exist $m,n \in \bbN$ such that $\lambda^m(X) = \lambda^n(Y)$, where 
	\[
	\lambda : \sum\limits_{n = 0}^\infty A_n\sigma^n \mapsto \sum\limits_{n = 0}^\infty A_{n+1}\sigma^n
	\]
	is the left-shift operator. We call the tail-equivalence class $[X]$ of $X$ the \defi{tail} of $X$.
\end{defn}

Informally, two series are tail-equivalent if they can be made equal by removing a finite set of terms from each; moreover, tail-equivalence is the symmetric closure of the relation
\[
\exists \ n \in \bbN, F \in R[\sigma] \ : \ X = F +\sigma^nY.
\]
This formulation is useful because the operator $\lambda$ does not have an internal representation in $R[[\sigma]]$.

\begin{theorem}\label{TailPoly}
If $P(T)$ is a $\Sum$-minimal polynomial for $X$, and $F \coloneqq \sum_{j=0}^{n-1}X_j \sigma^j$, then $P'(T) \coloneqq P(F+\sigma^n T)$ is a $\Sum$-minimal polynomial for $\lambda^n(X)$.  Conversely, if $Q(T) = \sum_{j=0}^m Q_j T^j$ is a $\Sum$-minimal polynomial for $\lambda^n(X)$, then 
$Q\prm(T) \coloneqq \sum_{j=0}^m \sigma^{n(m-j)} Q_j (T-F)^j$
is a $\Sum$-minimal polynomial for $X$.
\end{theorem}

\begin{proof}
 For ease of notation, write $Y \coloneqq \lambda^n(X)$. Clearly 
$$F+\sigma^n Y = \sum_{j=0}^{n-1}X_j\sigma^j + \sum_{j=n}^{\infty}X_j\sigma^j = X,$$
so $P'(Y) = P(F+\sigma^n Y) = P(X) = 0$. Conversely,
$$Q\prm(X) = \sum_{j=0}^m \sigma^{n(m-j)} Q_j (X-F)^j = \sigma^{nm} Q(Y) = \sigma^{nm} Q(Y)=0.$$
	Now write $\Sum (P)(t) = a \cdot \scalpoly X (t)$ with $a \in \Codomain^\times$, and $\Sum (Q)(t) = b \cdot \scalpoly Y (t)$ with $b \in \Codomain^\times$. Then $\Sum (P\prm)(t) = a \cdot \scalpoly X (x\prm + t)$ divides $\scalpoly Y (t)$ and $\Sum (Q\prm) (t) = b \cdot \scalpoly Y (t - x\prm)$ divides $\scalpoly X (t)$. We conclude that $\scalpoly X (x\prm + t)$ divides $\scalpoly Y (t)$, which in turn divides $\scalpoly X (x\prm + t)$, and so $\scalpoly X (x\prm + t) = \scalpoly Y (t)$ as desired. Now as $P\prm(Y) = 0$ and $\Sum (P\prm) (t)$ is a nonzero multiple of $\scalpoly Y (t)$, we conclude that $P\prm(T)$ is a $\Sum$-minimal polynomial for $Y$, and analogously $Q\prm(T)$ is a $\Sum$-minimal polynomial for $X$.
\end{proof}

\begin{cor}\label{Corollary: Tail-equivalence respects Sum-***}
	Let $X$ and $Y$ be tail-equivalent series. 
	\begin{enumerate}[label=(\alph*)]
		\item If $X$ is $\Sum$-transcendental, then $Y$ is also $\Sum$-transcendental.
		\item If $X$ is (absolutely) $\Sum$-algebraic, then $Y$ is also (absolutely) $\Sum$-algebraic, and $X$ and $Y$ have both the same $\Sum$-degree and the same scalar degree.
		\item If $X$ is $\Sum$-infinite, then $Y$ is also $\Sum$-infinite.
	\end{enumerate}
\end{cor}

Thus, it makes sense to call tails (absolutely) $\Sum$-algebraic, $\Sum$-transcendental, or $\Sum$-infinite, and, in the first case, to refer to the $\Sum$-degree or scalar degree of a tail.

\section{Absolutely $\Sum$-Algebraic Series}\label{Section: Absolutely Algebraic Series}

Let us take a closer look at the set $\AbsAlg \Sum$ of absolutely $\Sum$-algebraic series. We begin this section with a corrected version of Proposition 2.5i from \cite{Dawson}.

\begin{theorem}\label{Theorem: Ring-structure of Absolutely Sum-Algebraic Series}
	For every summation $\Sum \in \MSums \R \Codomain$, the set $\AbsAlg \Sum \subseteq \R[[\sigma]]$ is an $\R$-algebra.
\end{theorem}

\begin{proof}	
	Clearly $0$ and $1$ are absolutely $\Sum$-algebraic series. Now let $X$ and $Y$ be any absolutely $\Sum$-algebraic series, and let $\Sum\prm \in \MSums \R \Codomain$ be any extension of $\Sum$. We consider the set $\Zeroes {X + Y} {\Sum\prm}$, defined as in (\ref{Characterizing Set for Sum-*** Series}). Let $x$ be a root of the scalar polynomial $\scalpoly X \prm (t)$ for $X$ over $\Sum\prm$; by Theorem \ref{Theorem: Summation-structure of Sum-*** Series}, $\Sum\prm$ has an extension $\Sum^{\prime \prime} \in \MSums \R \Codomain$ such that $\Sum^{\prime \prime} (X) = x$. Now as $Y$ is absolutely $\Sum$-algebraic, we may choose a root $y$ of the scalar polynomial $\scalpoly Y ^{\prime \prime} (t)$ for $Y$ over $\Sum^{\prime \prime}$; applying Theorem \ref{Theorem: Summation-structure of Sum-*** Series} again, we may choose a summation $\Sum^{\prime \prime \prime} \in \MSums \R \Codomain$ extending $\Sum^{\prime \prime}$ for which $\Sum^{\prime \prime \prime}(Y) = y$. Then 
	\[
		\Sum^{\prime \prime \prime}(X + Y) = x + y \in \Codomain,
	\]
	and so $\Zeroes {X + Y} {\Sum\prm}$ is nonempty. On the other hand, since every summation summing $X + Y$ may be extended to sum $X$ and $Y$ as well, it is clear that 
	\[
		\Zeroes {X + Y} {\Sum\prm} \subseteq \set{x + y \ : \ x \in \Zeroes X \Sum, y \in \Zeroes Y \Sum}.
	\]
	Thus, as both of these latter sets are finite, we see that $\Zeroes {X + Y} {\Sum\prm}$ is finite as well. Then by Corollary \ref{Corollary: Alternate Characterization for Sum-*** Series}, the series $X + Y$ is $\Sum\prm$-algebraic for every $\Sum\prm \in \MSums \R \Codomain$ extending $\Sum$, and so $X + Y$ is absolutely $\Sum$-algebraic. A completely similar argument shows that $XY$ is absolutely $\Sum$-algebraic.
\end{proof}

We also have the following proposition, which follows from Theorem \ref{Theorem: Summation-structure of Sum-*** Series}.

\begin{prop}\label{Proposition: The absolutely algebraic series are the intersection of all superdomains}
	We have
	\[
	\AbsAlg \Sum = \bigcap \set{\Domain\prm \subseteq \R[[\sigma]] \ : \ (\Domain, \Sum) \subseteq (\Domain\prm, \Sum\prm)}.
	\]
\end{prop}

Theorem \ref{Theorem: Ring-structure of Absolutely Sum-Algebraic Series} and Proposition \ref{Proposition: The absolutely algebraic series are the intersection of all superdomains} both suggest that the absolutely $\Sum$-algebraic series are worthy of further study. Definition \ref{Definition: Sum-*** series} classifies a series by means of its scalar polynomial as either $\Sum$-transcendental, $\Sum$-algebraic, and $\Sum$-infinite. Determining whether a $\Sum$-algebraic series is absolutely $\Sum$-algebraic is a much subtler affair. We will require several intermediate lemmas before we can provide our answer (Theorem \ref{Theorem: Characterizations of absolutely Sum-algebraic series}). 

\begin{lemma}\label{Lemma: Existence of generic minimal polynomials for sums and products}
	Let $m$ and $n$ be positive integers, and let $\alpha_0, \ldots, \alpha_{m-1}, \beta_0, \ldots, \beta_{n-1}$ be indeterminates over $\bbZ$. There are unique polynomials
	\[
		P_{m,n}(T; \alpha_0, \ldots, \alpha_{m-1}, \beta_0, \ldots, \beta_{n-1}) \in \bbZ[\alpha_0, \ldots, \alpha_{m-1}, \beta_0, \ldots, \beta_{n-1}][T]
	\]
	 and 
	 \[
		 Q_{m,n}(T; \alpha_0, \ldots, \alpha_{m-1}, \beta_0, \ldots, \beta_{n-1}) \in \bbZ[\alpha_0, \ldots, \alpha_{m-1}, \beta_0, \ldots, \beta_{n-1}][T]
	 \] 
	 satisfying the following pair of conditions:
	\begin{enumerate}
		\item $P_{m,n}(T; \alpha_0, \ldots, \alpha_{m-1}, \beta_0, \ldots, \beta_{n-1})$ and $Q_{m,n}(T; \alpha_0, \ldots, \alpha_{m-1}, \beta_0, \ldots, \beta_{n-1})$ are monic of degree $mn$ as polynomials in $T$.\label{Monicity and degree}
		\item For any commutative ring $\R\prm$ and any elements $X, Y, A_0, \ldots, A_{m-1}, \beta_0, \ldots, \beta_{n-1} \in \R\prm$ such that \label{Universal minimal polynomials}
		\begin{align*}
			X^m + \sum\limits_{k < m} A_k X^k = Y^n + \sum\limits_{k < n} B_k Y^k &= 0,
		\end{align*}
		we have \begin{align*}
			P_{m, n}(X + Y; A_0, \ldots, A_{m-1}, B_0, B_{n-1}) &= 0 \\
			\intertext{ and }
			Q_{m, n}(XY; A_0, \ldots, A_{m-1}, B_0, B_{n-1}) &= 0
		\end{align*}
		as elements of $R$.
	\end{enumerate}
\end{lemma}

\begin{proof}
	Let $\R = \bbZ[\alpha_0, \ldots, \alpha_{m-1}, \beta_0, \ldots, \beta_{n-1}]$ where $\alpha_i$ and $\beta_j$ are indeterminate, let 
	\begin{align*}
	f(T) &\coloneqq T^m + \sum\limits_{k < m} \alpha_k T^k, \ \text{and let} \\
	g(T) &\coloneqq T^n + \sum\limits_{k < n} \beta_k T^k.
	\end{align*}
	Now take ${\rm S} \coloneqq \R[\chi, \upsilon] / (f(\chi), g(\upsilon))$; ${\rm S}$ is an integral domain. We let $P_{m,n}(T)$ be the minimal polynomial for $\chi + \upsilon$ over $\R$, and let $Q_{m, n}(T)$ be the minimal polynomial for $\chi \upsilon$ over $R$. More explicitly, let ${\rm S}\prm$ be the splitting field for $f(T)$ and $g(T)$ over ${\rm S}$, let $\chi \eqqcolon \chi_1, \chi_2, \ldots, \chi_m$ be the roots of $f(T)$, and let $\upsilon \eqqcolon \upsilon_1, \upsilon_2, \ldots, \upsilon_n$ be the roots of $g(T)$. Then
	\begin{align*}
		P_{m,n}(T; \alpha_0, \ldots, \alpha_{m-1}, \beta_0, \ldots, \beta_{n-1}) &= \prod\limits_{i = 1}^m \prod\limits_{j = 1}^n (T - \chi_i - \upsilon_j), \text{ and } \\
		Q_{m, n}(T; \alpha_0, \ldots, \alpha_{m-1}, \beta_0, \ldots, \beta_{n-1}) &= \prod\limits_{i = 1}^m \prod\limits_{j = 1}^n (T - \chi_i \upsilon_j).
	\end{align*} 
	This shows that $P_{m, n}(T)$ and $Q_{m, n}(T)$	are both of degree $mn$ in $T$.
	
	Now by construction, ${\rm S}$ satisfies the following universal property: for any commutative ring $\R\prm$ that contains elements $X, Y, A_0, \dots, A_{m-1}, B_0, \dots, B_{n-1}$ satisfying condition \eqref{Universal minimal polynomials} of Lemma \ref{Lemma: Existence of generic minimal polynomials for sums and products}, there is a unique homomorphism $\varphi : {\rm S} \to \R\prm$ such that
	\[
	\varphi : \upsilon \mapsto X, \chi \mapsto Y, \alpha_i \mapsto A_i, \beta_j \mapsto B_j.
	\]
	Now as $\varphi(P_{m, n}(\chi + \upsilon)) = 0$ and $\varphi(Q_{m, n}(\chi \upsilon) = 0$, the polynomials $P_{m, n}$ and $Q_{m, n}$ satisfy properties \eqref{Monicity and degree} and \eqref{Universal minimal polynomials} of Lemma \ref{Lemma: Existence of generic minimal polynomials for sums and products}. On the other hand, as monic minimal polynomials in an integral domain are unique, and by the uniqueness of commutative rings satisfying universal properties, $P_{m, n}$ and $Q_{m, n}$ are the only polynomials with the asserted properties.
\end{proof}

The polynomials $P_{m,n}(T; \alpha_0, \ldots, \alpha_{m-1}, \beta_0, \ldots, \beta_{n-1})$ and $Q_{m,n}(T; \alpha_0, \ldots, \alpha_{m-1}, \beta_0, \ldots, \beta_{n-1})$ are effectively computable \cite{MacDuffee}[Section 59].

\begin{lemma}\label{Lemma: generic minimal polynomials evaluate to t^mn at zero}
	Let $P_{m,n}(T; \alpha_0, \ldots, \alpha_{m-1}, \beta_0, \ldots, \beta_{n-1})$ and $Q_{m,n}(T; \alpha_0, \ldots, \alpha_{m-1}, \beta_0, \ldots, \beta_{n-1})$ be as in Lemma \ref{Lemma: Existence of generic minimal polynomials for sums and products}. Then 
	\[
		P_{m, n}(T; 0, \ldots, 0) = Q_{m, n}(T; 0, \ldots, 0) = T^{mn}.
	\]
\end{lemma}

\begin{proof}
	We follow the notation in the proof of Lemma \ref{Lemma: Existence of generic minimal polynomials for sums and products}. Note that $\chi_1, \ldots, \chi_m, \upsilon_1, \ldots, \upsilon_n$ are algebraically independent over $\bbZ$, and so by the universal mapping property of polynomial rings, there exists a map $\varphi : {\rm S}\prm \to \bbZ$ with $X_i \mapsto 0$ for all $1 \leq i \leq m$, and $Y_j \mapsto 0$ for all $1 \leq j \leq n$. But by the theory of symmetric polynomials, we see that $X_1, \ldots, X_m \mapsto 0$ if and only if $\alpha_0, \ldots, \alpha_{m - 1} \mapsto 0$, and $\upsilon_1, \ldots, \upsilon_n \mapsto 0$ if and only if $\beta_0, \ldots, \beta_{n-1} \mapsto 0$. Then 
	\begin{align*}
		P_{m,n}(T; 0, \ldots, 0) &= \prod\limits_{i = 1}^m \prod\limits_{j = 1}^n (T - 0 - 0) = T^{mn}, \ \text{and} \\
		Q_{m, n}(T; 0, \ldots, 0) &= \prod\limits_{i = 1}^m \prod\limits_{j = 1}^n (T - 0 \cdot 0) = T^{mn},
	\end{align*} 
	as desired.
\end{proof}

If $X$ and $Y$ are $\Sum$-infinite, then intuitively, it makes sense for their product to be $\Sum$-infinite as well. This is true (Corollary \ref{Corollary: a product of infinite series is infinite}), but we first prove a weaker claim.

\begin{lemma}\label{Lemma: product of infinite series are infinite or zero}
	If $X$ and $Y$ are $\Sum$-infinite, then $\scalpoly{XY} (T) = t^\ell$ for some $\ell \geq 0$. Equivalently, $\Zeroes {XY} \Sum \subseteq \set 0$.
\end{lemma}

\begin{proof}
	Let $X$ and $Y$ be $\Sum$-infinite. By Corollary \ref{Corollary: Tail-equivalence respects Sum-***}, we may assume without loss of generality that $X$ and $Y$ are units. Now let $P(T) = \sum\limits_{k = 0}^m P_k T^k$ be a $\Sum$-minimal polynomial for $X$, and let $Q(T) = \sum\limits_{k = 0}^n Q_k T^k$ be a $\Sum$-minimal polynomial for $Y$. Suppose first that $P_0 = Q_0 = 1$. In this case, we see that $\reflected{P}(T)$ and $\reflected{Q}(T)$ are monic, and $\Sum (\reflected{P})(t) = t^m$ and $\Sum (\reflected{Q}) (t) = t^n$. By Corollary \ref{Corollary: scalar polynomial of inverse divides reflected polynomial} and Lemma \ref{Lemma: Existence of generic minimal polynomials for sums and products}, we conclude $X\inv Y\inv$ is a root of $Q_{m, n}(T; P_m, P_{m-1} \ldots, P_1, Q_n, Q_{n-1}, \ldots, Q_1)$. Moreover, since $\Sum (P_k) = 0$ for $k < m$ and $\Sum Q_k = 0$ for $k < n$, we may apply Lemma \ref{Lemma: generic minimal polynomials evaluate to t^mn at zero} to conclude
	\[
		\Sum (Q_{m, n} (t;  P_m, P_{m-1} \ldots, P_1, Q_n, Q_{n-1}, \ldots, Q_0)) =  Q_{m, n} (t; 0, \ldots, 0) = t^{mn}.
	\]
	Then by Proposition \ref{Proposition: If P(X) = 0 then the scalar polynomial divides Sum P(X)}, we see $\scalpoly {X\inv Y\inv} (t)$ divides $t^{mn}$, and so $\scalpoly {X\inv Y\inv}$ is a power of $t$. Thus the sum of a $\Sum$-minimal polynomial for $X\inv Y\inv$ has exactly one nonzero term, and so the sum of its reflected polynomial will also only have one nonzero term. By Corollary \ref{Corollary: scalar polynomial of inverse divides reflected polynomial}, we conclude that the scalar polynomial for $XY$ also has only one nonzero term, and the lemma follows in this case.
	
	We now relax the condition that $P_0 = Q_0 = 1$. In any case, $P_0 Q_0 X\inv Y\inv$ is a root of the monic polynomial 
	\begin{align*}
		Q_{m,n}&(T; P_0^{m-1} P_m, P_0^{m-2} P_{m-1} \ldots, P_1, Q_0^{n-1} Q_n, Q_0^{n-2} Q_{n-1}, \ldots, Q_0),
		\intertext{and}
		\Sum (Q_{m,n}&(t; P_0^{m-1} P_m, P_0^{m-2} P_{m-1} \ldots, P_1, Q_0^{n-1} Q_n, Q_0^{n-2} Q_{n-1}, \ldots, Q_0)) = t^{mn}.
	\end{align*}
	As $X$ and $Y$ are $\Sum$-infinite, we must have $p \coloneqq \Sum (P_0) \neq 0$ and $q \coloneqq \Sum (Q_0) \neq 0$. Thus $\Zeroes {P_0 Q_0 X\inv Y\inv} \Sum \subseteq \set{0}$. Now suppose by way of contradiction that $\Zeroes {XY} \Sum \not\subseteq \set{0}$, and choose $a \neq 0$ an element of $\Zeroes {XY} \Sum$. By Theorem \ref{Theorem: Summation-structure of Sum-*** Series} and Theorem \ref{Theorem: Rational Extension}, we may choose a rationally closed extension $\Sum\prm$ of $\Sum$ with $\Sum (XY) = a$. As $\Sum\prm$ is rationally closed, we must have $\Sum\prm (X\inv Y\inv) = \frac{1}{a}$, and by multiplicativity $\Sum\prm (P_0 Q_0 X\inv Y\inv) = \frac{pq}{a} \neq 0$. Then by Corollary \ref{Corollary: Alternate Characterization for Sum-*** Series}, we have 
	\[
		0 \neq \frac{pq}{a} \in \Zeroes {P_0 Q_0 X\inv Y\inv} {\Sum\prm} \subseteq \Zeroes {P_0 Q_0 X\inv Y\inv} \Sum \subseteq \set{0},
	\] 
	which is a contradiction.
\end{proof}

\begin{defn}\label{Definition: Practically Zero}
	A series $X$ is \defi{practically $\Sum$-zero} if it is absolutely $\Sum$-algebraic and $\scalpoly X (t) = t^m$ for some $m > 0$.
\end{defn}

By Theorem \ref{Theorem: Summation-structure of Sum-*** Series}, Corollary \ref{Corollary: Scalar polynomials lose factors under extensions}, and Definition \ref{Definition: Absolutely Sum-algebraic}, a series $X$ is practically $\Sum$-zero if and only if every summation $\Sum\prm$ extending $\Sum$ has an extension $\Sum^{\prime \prime}$ such that $\Sum^{\prime \prime} (X) = 0$. See also Definitions \ref{Definition: Sum-univalence} and \ref{Definition: Univalent Extension} below.

\begin{lemma}\label{Lemma: Invertible Sum-infinite series}
	Let $U$ be a unit of $\R[[\sigma]]$. Then $U$ is $\Sum$-infinite if and only $U\inv$ is practically $\Sum$-zero.
\end{lemma}

\begin{proof}
	Suppose first that $U$ is $\Sum$-infinite, and let $P(T) = \sum\limits_{k = 0}^m P_k T^k$ be a $\Sum$-minimal polynomial for $U$ with $P_m \neq 0$. By Corollary \ref{Corollary: scalar polynomial of inverse divides reflected polynomial}, $\scalpoly {U\inv} (t)$ divides $\Sum (\reflected{P}) (t)$. But $\Sum (\reflected{P}) (t) = \Sum (P_0) t^m$, and so $U\inv$ is either $\Sum$-infinite or $\Sum$-algebraic with zero as its only root. Now suppose by way of contradiction that $U\inv$ is not practically $\Sum$-zero. Then there is some extension $\Sum\prm$ of $\Sum$ for which $U\inv$ is $\Sum\prm$-infinite. Now by Lemma \ref{Lemma: product of infinite series are infinite or zero}, we see that $U \cdot U\inv = 1$ has a scalar polynomial over $\Sum\prm$ of the form $t^m$ for some $m$, which is absurd. Thus $U\inv$ is practically $\Sum$-zero as desired.
		
	On the other hand, suppose that $U\inv$ is practically $\Sum$-zero, and suppose by way of contradiction that $U$ is not $\Sum$-infinite. Then $U$ is $\Sum$-algebraic or $\Sum$-transcendental, and so by Theorem \ref{Theorem: Summation-structure of Sum-*** Series}, we may find $x \in \Codomain$ and $\Sum\prm$ an extension of $\Sum$ such that $\Sum\prm (U) = x$. On the other hand, as $U\inv$ is practically $\Sum$-zero, there is an extension $\Sum^{\prime \prime}$ of $\Sum\prm$ such that $\Sum^{\prime \prime} (U\inv) = 0$. Then 
	\[
		1 = \Sum^{\prime \prime} (U \cdot U\inv) = \Sum^{\prime \prime} (U) \cdot \Sum^{\prime \prime} (U\inv) = x \cdot 0 = 0,
	\]
	which is a contradiction.
\end{proof}

\begin{example} The series $Y$ of Example \ref{AddInfinite} is the reciprocal of the series $X$ of Example \ref{PracAddZero}; its $\Add$-minimal polynomial is $1$, so it is $\Add$-infinite.
\end{example}

\begin{proposition}\label{Proposition: Classification of inverses}
	Let $U \in \R[[\sigma]]$ be a unit. Then
	\begin{enumerate}[label=(\alph*)]
		\item The series $U$ is $\Sum$-transcendental if and only if $U\inv$ is $\Sum$-transcendental. \label{Inverse of Sum-transcendental series}
		\item The series $U$ is (absolutely) $\Sum$-algebraic if and only if $U\inv$ is either $\Sum$-infinite or (absolutely) $\Sum$-algebraic but not practically $\Sum$-zero.\label{Inverse of Sum-algebraic series}
		\item The series $U$ is $\Sum$-infinite if and only if $U\inv$ is practically $\Sum$-zero. \label{Inverse of Sum-infinite series}
	\end{enumerate}
\end{proposition}

\begin{proof}
	Let $U \in \R[[\sigma]]$ be a unit. We apply Corollary \ref{Corollary: Alternate Characterization for Sum-*** Series}. For each $x \in \Zeroes U \Sum$, there exists a multiplicative summation $\Sum_x$ such that $\Sum_x(U) = x$. Replacing $\Sum_x$ with its rational closure if necessary, we may take $\Sum_x$ to be rationally closed. Then if $x \neq 0$, we see $\Sum_x (U\inv) = x\inv$. \textit{A fortiori}, we conclude 
	\[
	\set{x\inv \ : \ x \in \Zeroes U \Sum, \ x \neq 0} \subseteq \Zeroes {U\inv} \Sum. 
	\]	
Suppose first that $U$ is $\Sum$-transcendental. The field $\Codomain$ is algebraically closed, and so is infinite. Thus $\Zeroes {U\inv} \Sum \supseteq \Codomain^\times$ is infinite, so by Corollary \ref{Corollary: Alternate Characterization for Sum-*** Series} the series $U\inv$ is $\Sum$-transcendental. By symmetry, if $U\inv$ is $\Sum$-transcendental then $U$ is $\Sum$-transcendental.

Now \ref{Inverse of Sum-infinite series} is just Lemma \ref{Lemma: Invertible Sum-infinite series}, and \ref{Inverse of Sum-algebraic series} follows from \ref{Inverse of Sum-transcendental series}, \ref{Inverse of Sum-infinite series}, and the partition (\ref{Partition of R[[sigma]]}).
\end{proof}

At last, we are ready to distinguish which $\Sum$-algebraic series are absolutely $\Sum$-algebraic.

\begin{theorem}\label{Theorem: Characterizations of absolutely Sum-algebraic series}
	Let $X$ be a $\Sum$-algebraic series. The following are equivalent:
	\begin{enumerate}[label=(\alph*)]
		\item The series $X$ is absolutely $\Sum$-algebraic.\label{Condition: X is absolutely Sum-algebraic}
		\item For some unit $U$ tail-equivalent to $X$, we have $\scalpoly {U\inv} (0) \neq 0$. \label{Condition: Exists unit with s(0) != 0}	
		\item For every unit $U$ tail-equivalent to $X$, we have $\scalpoly {U\inv} (0) \neq 0$. \label{Condition: For all units, s(0) != 0}
	\end{enumerate}
\end{theorem}

\begin{proof}
	$\ref{Condition: X is absolutely Sum-algebraic} \implies \ref{Condition: Exists unit with s(0) != 0}$: Let $X$ be an absolutely $\Sum$-algebraic series, and let $U \coloneqq 1 - \sigma + \sigma^2 X$. As $U \in 1 + \sigma \R[[\sigma]] \subseteq \R[[\sigma]]^\times$, we see $U$ is a unit, which is manifestly tail-equivalent to $X$. Thus by Corollary \ref{Corollary: Tail-equivalence respects Sum-***}, $U$ is absolutely $\Sum$-algebraic. But now suppose by way of contradiction that $\scalpoly {U\inv} (0) = 0$. Then by Theorem \ref{Theorem: Summation-structure of Sum-*** Series}, there is a multiplicative summation $\Sum\prm$ extending $\Sum$ such that $\Sum\prm(U\inv) = 0$. In particular, $U\inv$ is practically $\Sum\prm$-zero, and so $U$ is $\Sum\prm$-infinite by Lemma \ref{Lemma: Invertible Sum-infinite series}. But as $U$ is absolutely $\Sum$-algebraic, Proposition \ref{Proposition: Every Sum-*** Series is a Sum'-*** Series} tells us $U$ is absolutely $\Sum\prm$-algebraic, and we obtain a contradiction.
	
	$\ref{Condition: Exists unit with s(0) != 0} \implies \ref{Condition: For all units, s(0) != 0}$: We proceed by contrapositive. Suppose that $V$ is a unit tail-equivalent to $X$ with $\scalpoly {V\inv} (0) = 0$. By Theorem \ref{Theorem: Summation-structure of Sum-*** Series}, there exists a summation $\Sum\prm$ extending $\Sum$ such that $\Sum\prm (V\inv) = 0$. \textit{A fortiori}, we see $V\inv$ is practically $\Sum\prm$-zero, and so by Lemma \ref{Lemma: Invertible Sum-infinite series}, we see $V$ is $\Sum\prm$-infinite. Then by Corollary \ref{Corollary: Tail-equivalence respects Sum-***}, if $U$ is a unit is tail-equivalent to $V$, then $U$ is $\Sum\prm$-infinite; equivalently, if $U$ is a unit tail-equivalent to $X$, then $U$ is $\Sum\prm$-infinite. Applying Lemma \ref{Lemma: Invertible Sum-infinite series} again, we see that such a $U\inv$ is practically $\Sum\prm$-zero. Then by Corollary \ref{Corollary: Scalar polynomials lose factors under extensions}, we conclude $\scalpoly {U\inv} (0) = 0$, which proves our claim.	
	
	$\ref{Condition: For all units, s(0) != 0} \implies \ref{Condition: X is absolutely Sum-algebraic}$: Suppose by contrapositive that $X$ is not absolutely $\Sum$-algebraic, and let $U \coloneqq 1 - \sigma + \sigma^2 X$. If $X$ is $\Sum$-transcendental, then by Corollary \ref{Corollary: Tail-equivalence respects Sum-***}, so is $U$, and so by Proposition \ref{Proposition: Classification of inverses}, $U\inv$ is $\Sum$-transcendental. Then $\scalpoly {U\inv} (t) = 0$, and in particular $\scalpoly {U\inv} (0) = 0$, as desired. Suppose now that $X$ is neither absolutely $\Sum$-algebraic nor $\Sum$-transcendental. Then $X$ is $\Sum\prm$-infinite for some extension $\Sum\prm$ of $\Sum$; by Corollary \ref{Corollary: Tail-equivalence respects Sum-***}, $U$ is also $\Sum\prm$-infinite. Now by Corollary \ref{Lemma: Invertible Sum-infinite series}, $U\inv$ is effectively $\Sum\prm$-zero, so the scalar polynomial for $U\inv$ with respect to $\Sum\prm$ is a positive power of $t$. Then by Corollary \ref{Corollary: Scalar polynomials lose factors under extensions}, that same power of $t$ divides $\scalpoly {U\inv} (t)$, and so $\scalpoly {U\inv} (0) = 0$. Our assertion holds by contrapositive.	
\end{proof}

\begin{corollary}\label{Corollary: A unit is sum-algebraic if and only if s_U inv(0) neq 0}
	A unit $U$ of $\R[[\sigma]]$ is absolutely $\Sum$-algebraic if and only if $\scalpoly {U\inv} (0) \neq 0$.
\end{corollary}

We have another sufficient condition for a series $X$ to be absolutely $\Sum$-algebraic.

\begin{corollary}\label{Corollary: sum-degree suffices for being absolutely sum-algebraic}
	If $X$ is a $\Sum$-algebraic series with $\Sum$-degree equal to its scalar degree, then $X$ is absolutely $\Sum$-algebraic.
\end{corollary}

\begin{proof}
	Let $U = 1 - \sigma + \sigma^2 X$; by Corollary \ref{Corollary: Tail-equivalence respects Sum-***}, $U$ is absolutely $\Sum$-algebraic if and only if $X$ is absolutely $\Sum$-algebraic. Now suppose that that the $\Sum$-degree of $U$ equals its scalar degree, and let $P(t)$ be a $\Sum$-minimal polynomial for $U$ with degree equal to its $\Sum$-degree. Then $\reflected P (U\inv) = 0$, and so $\scalpoly {U\inv} (t)$ divides $\Sum (\reflected P) (t)$. But observe that $\reflected P (0) \neq 0$, hence $\scalpoly {U\inv} (0) \neq 0$. Then by Theorem \ref{Theorem: Characterizations of absolutely Sum-algebraic series}, our claim follows.
\end{proof}

\begin{example}\label{PracAddZero} Let $X \in \bbC[\sigma]$ be the the Taylor series for $\sqrt{1-\sigma}$, so
\[
X \coloneqq 1 - \frac{\sigma}2 - \frac{\sigma^2}8 - \frac{\sigma^3}{16} - \frac{5\sigma^4}{128} - \frac{7\sigma^5}{256} - \cdots.
\]
We see $X$ has minimal polynomial $T^2-(1-\sigma)$ and scalar polynomial $t^2$. It is thus $\Add$-algebraic; indeed, by Corollary \ref{Corollary: sum-degree suffices for being absolutely sum-algebraic}, $Y$ is absolutely $\Sum$-algebraic. As $\Zeroes X \Sum = \set 0$, we conclude that $X$ is practically $\Sum$-zero.
\end{example} 

Although Corollary \ref{Corollary: sum-degree suffices for being absolutely sum-algebraic} provides a sufficient criterion for verifying that a series is absolutely $\Sum$-algebraic, we have been unable to determine if this condition is also necessary. Thus, we pose the following question.

\begin{quest}
	Suppose $\Sum$ is a multiplicative summation to an algebraically closed field $\Codomain$. If $X$ is an absolutely $\Sum$-algebraic series, must it have $\Sum$-degree equal to its scalar degree?
\end{quest}

\section{Univalent Extensions}\label{Section: Univalent Extensions}

Recall that $\Codomain$ is an algebraically closed field. Although $\Q\Sum$ is multiplicatively canonical, it is generally not the multiplicative fulfillment of $\Sum$. Consider for instance the series $X$ of Example \ref{PracAddZero}. This series was shown to be practically $\Add$-zero; it follows, by the remarks after Definition \ref{Definition: Practically Zero}, that $X$ is in the multiplicative fulfillment of $\Add$. On the other hand, $X$ is not a rational function in $\sigma$, and so is not an element of $\Q\R[\sigma]$. This example is suggestive: if the scalar polynomial of an absolutely $\Sum$-algebraic series $X$ is not linear but still only has one root, this forces $X$ to be summed to a unique value, and places it in the multiplicative fulfillment of $\Sum$. In this section, we follow this intuition to construct the multiplicative fulfillment of $\Sum$.

\begin{defn}\label{Definition: Sum-univalence}
	If $X$ is $\Sum$-algebraic and $\scalpoly X (t) = (t-\rho)^m$ for some $\rho \in \Codomain$ and $m > 0$, we say $X$ is \defi{$\Sum$-univalent} with root $\rho \eqqcolon \scalroot X$. If in addition $X$ is absolutely $\Sum$-algebraic, we say $X$ is \defi{absolutely $\Sum$-univalent with root $\scalroot X$}.
\end{defn}
A series $X$ is practically $\Sum$-zero exactly if it is absolutely $\Sum$-univalent with root $\scalroot X = 0$; thus, Definition \ref{Definition: Sum-univalence} provides a natural extension of Definition \ref{Definition: Practically Zero}. 

\begin{example}
	Let $Z$ be the Taylor series of ${{3 - \sigma + \sqrt{1 - 6 \sigma + 5 \sigma^2}} \over 2}$; thus we have
	\[
		Z = 2 - 2\sigma - \sigma^2 - 3\sigma^3 - 10\sigma^4 - 36\sigma^5 - \ldots.
	\]
	We claim $Z$ is $\Add$-univalent. Indeed, $Z$ has $\Sum$-minimal polynomial
\[
	P(T) = T^2 - (3-\sigma) T + (2-\sigma^2) \in \Domain[t];
\]
therefore its scalar polynomial is 
\[
	\scalpoly Z (t) = t^2 - 2t + 1 = (t-1)^2 
\]
and so $Z$ is $\Add$-univalent with root $\scalroot Z = 1$. Corollary \ref{Corollary: sum-degree suffices for being absolutely sum-algebraic} confirms that $Z$ is absolutely $\Sum$-univalent.
\end{example}

The series $Y$ described in Example \ref{Not absolutely Sum-algebraic} furnishes an example of a series which is $\Sum$-univalent but not absolutely $\Sum$-univalent.

Note that if $X$ is absolutely $\Sum$-univalent, then by Corollary \ref{Corollary: Scalar polynomials lose factors under extensions} and Proposition \ref{Proposition: Every Sum-*** Series is a Sum'-*** Series}, it is absolutely $\Sum\prm$-univalent for every $\Sum\prm \in \MSums \R \Codomain$ extending $\Sum$ and the root of $X$ is invariant under extensions of the summation.
\begin{defn}\label{Definition: Univalent Extension}
	For any multiplicative summation $(\Domain, \Sum) \in \MSums \R \Codomain$, the \defi{univalent extension} $(\U\Domain, \U\Sum)$ of $(\Domain, \Sum)$ is defined as follows. For $X \in \R[[\sigma]]$, we say $X \in \U\Domain$ if $X$ is absolutely $\Sum$-univalent. We define
	\begin{align*}
		\U\Sum &: \U\Domain \to \Codomain, \\ 
		\U\Sum &: X \mapsto \scalroot X \ \text{if} \ X \ \text{is as above.}
	\end{align*}
	We extend this definition to weakly multiplicative summations $(\Domain, \Sum) \in \WSums \R \Codomain$ by setting $(\U\Domain, \U\Sum) \coloneqq (\U\M\Domain, \U\M\Sum)$.
\end{defn}

\begin{lemma}\label{Lemma: The image of U(Sum) is a purely inseparable extension of Sum(Domain)}
	For any multiplicative summation $(\Domain, \Sum)$, $(\U\Domain, \U\Sum)$ is a well-defined multiplicative summation. Moreover, the image of $\U\Domain$ under $\U\Sum$ is a purely inseparable extension of the field of fractions of $\Sum(\Domain)$. In particular, if $\Codomain$ is of characteristic 0, then the image of $\U\Domain$ under $\U\Sum$ is the field of fractions of $\Sum(\Domain)$.
\end{lemma}

\begin{proof}
	The map $\U\Sum$ is well-defined as a function because $\scalpoly X (t)$ is well-defined for every $\Sum$-algebraic $X \in \R[[\sigma]]$. If $X$ and $Y$ are absolutely $\Sum$-univalent, then by Theorem \ref{Theorem: Ring-structure of Absolutely Sum-Algebraic Series} we see $X+Y$ and $XY$ are absolutely $\Sum$-algebraic. Then as $\Codomain$ is algebraically closed, we see
	\[
		\emptyset \subsetneq \Zeroes {X + Y} \Sum \subseteq \set{x + y \ : \ x \in \Zeroes X \Sum, y \in \Zeroes Y \Sum} = \set{\scalroot X + \scalroot Y},
	\] 
	and $X + Y$ is absolutely $\Sum$-univalent with root $\scalroot X + \scalroot Y$. A similar argument holds for $XY$. Thus $\U\Sum$ is a multiplicative summation, and the image of $\U\Sum$ is at least an $\R$-algebra. 
	
	The proof that the image of $\U\Sum$ is a field proceeds along the same lines as in the proof of Proposition \ref{Proposition: The image of Q(Sum) is the field of fractions of Sum(Domain)}. Let $x \neq 0$ be in the image of $\U\Sum$, and choose $X \in \U\Domain$ with $\U\Sum (X) = x$. Replacing $X$ with $1 - \sigma + \sigma^2 X$ if necessary, we may assume that $X$ is a unit. By assumption $\scalpoly X (0) \neq 0$, then by Theorem \ref{Theorem: Characterizations of absolutely Sum-algebraic series}, we see $X\inv$ is absolutely $\Sum$-algebraic. But now let $P(T) = \sum\limits_{k = 0}^\ell P_k T^k$ be a $\Sum$-minimal polynomial for $X$, and suppose $\Sum (P)(t) = a \cdot (t - x)^m$ with $a \in \Codomain^\times$. Note $\Sum (\reflected{P}) (t) = a (-x)^m \cdot t^{\ell - m}(t - x\inv)^m$, then by Corollary \ref{Corollary: scalar polynomial of inverse divides reflected polynomial}, we see $\scalpoly {X\inv} (t)$ divides $t^{\ell - m} (t - x\inv)^m$. Applying Theorem \ref{Theorem: Characterizations of absolutely Sum-algebraic series} again and recalling that $X$ is absolutely $\Sum$-algebraic, we see that $\scalpoly {X\inv}(0) \neq 0$, and so $X\inv$ is absolutely $\Sum$-univalent with root $x\inv$. It remains to show that this field is a totally inseparable extension of the field of fractions of $\Sum(\Domain)$.
	
	Let $X$ be $\Sum$-univalent with root $\scalroot X$. Then there exists a $\Sum$-minimal polynomial $P(T) \in \Domain[t]$ with $P(X) = 0$ and $\Sum (P)(t) = a \cdot (t - \scalroot X)^m$ for some $a \in \Codomain^\times$ and $m > 0$ an integer. If $\Codomain$ is of characteristic 0 then 
	\[
		\scalroot X = - \frac{\Sum (P_{m-1})}{m \Sum (P_m)}
	\]
	is an element of the field of fractions of $\Sum(\Domain)$ as desired. Otherwise, let $p$ prime be the characteristic of $\Codomain$, and write $m = p^\mu \cdot m\prm$ with $(p,m\prm) = 1$. Then 
	\[
		\scalroot X^{p^\mu} = - \frac{\Sum (P_{p^\mu \cdot (m\prm-1)})}{m\prm \Sum(P_m)}
	\]
	and so $\scalroot X$ is a $p^\mu$th root of an element of the field of fractions of $\Sum(\Domain)$, hence purely inseparable over the field of fractions of $\Sum(\Domain)$. The claim follows.
\end{proof}

\begin{theorem}\label{Theorem: Univalent Extension}
		The map of multiplicative summations 
		\[
			\U : \MSums \R \Codomain \to \MSums \R \Codomain
		\] 
		is a well-defined extension map. Moreover, $\U$ is idempotent, multiplicatively canonical, preserves multiplicative compatibility, and subsumes $\T$ and $\Q$.
\end{theorem}

\begin{proof}
	We showed in Lemma \ref{Lemma: The image of U(Sum) is a purely inseparable extension of Sum(Domain)} that $\U\Sum$ is a well-defined multiplicative summation. Suppose now that $X \in \Domain$, and define $P(T) = T - X \in \Domain[T]$. Clearly $P(X) = 0$ and 
	\[
		\Sum (P)(t) = \scalpoly X (t) = t - \Sum X,
	\] 
	then $\U\Sum (X) = \Sum (X)$. Hence $\U\Sum$ is well-defined, multiplicative, and an extension of $\Sum$.
				
	Suppose $\Sum\prm$ extends $\Sum$. By Proposition \ref{Proposition: Every Sum-*** Series is a Sum'-*** Series} we have $\AbsAlg {\Sum\prm} \supseteq \AbsAlg \Sum$, and so by Corollary \ref{Corollary: Scalar polynomials lose factors under extensions}, if $X$ is absolutely $\Sum$-univalent then $X$ is absolutely $\Sum\prm$-univalent. Then $\U(\Sum\prm)$ extends both $\Sum\prm$ and $\U\Sum$, and so $\U\Sum$ is multiplicatively $\Sum$-canonical.
	
	Now suppose that $X \in \U\U\Domain$ with root $\rho$. As $\U\U\Sum$ is a multiplicatively $\Sum$-canonical extension of $\Sum$, Theorem \ref{Theorem: Summation-structure of Sum-*** Series} tells us that the scalar polynomial $\scalpoly X (t)$ for $X$ has only one root in $\Codomain$, and \textit{a fortiori} that $X$ is $\Sum$-algebraic. Then as $\Codomain$ is algebraically closed, we see $X$ is $\Sum$-univalent. Again, as $\U\U\Sum$ is multiplicatively canonical, $X$ must be absolutely $\Sum$-algebraic, as otherwise there would be an extension $\Sum\prm$ of $\Sum$ over which $X$ would be $\Sum\prm$-infinite. Then $X$ is absolutely $\Sum$-univalent, and $\U$ is idempotent as desired.
		
	Recall that $\U$ preserves multiplicative compatibility if and only if it preserves extensions. But if $\Sum\prm$ extends $\Sum$, and $X$ is $\Sum$-univalent, then by Corollary \ref{Corollary: Scalar polynomials lose factors under extensions}, $X$ is also $\Sum\prm$-univalent. Then $U(\Sum\prm)$ extends $\U\Sum$ as desired.
	
	Finally, if $X \in \Q\Domain$ then $X \in \U\Domain$, hence by the idempotence of $\U$ we have 
	\[
		\U\Domain = \U\U\Domain = \Q\U\Domain = \U\Q\Domain,
	\]
	and so $\U$ subsumes $\Q$. But then as $\Q$ subsumes $\T$, we see $\U$ subsumes $\T$ as well.
\end{proof}

\begin{cor}\label{Corollary: a product of infinite series is infinite}
	If $X$ and $Y$ are $\Sum$-infinite, then $XY$ is $\Sum$-infinite.
\end{cor}

\begin{proof}
	Without loss of generality, assume that $X$ and $Y$ are units. By Proposition \ref{Proposition: Classification of inverses}, we see that $X\inv$ and $Y\inv$ are practically $\Sum$-zero, and so by Theorem \ref{Theorem: Univalent Extension}, $X\inv Y\inv$ is also practically $\Sum$-zero. Another application of Proposition \ref{Proposition: Classification of inverses} shows that $(X\inv Y\inv)\inv = X Y$ is $\Sum$-infinite, as desired.
\end{proof}

\begin{theorem}\label{Theorem: Univalent Extension is multiplicative fulfillment}
	For any summation $(\Domain, \Sum) \in \MSums \R \Codomain$, $(\U\Domain, \U\Sum)$ is the multiplicative fulfillment of $\Sum$.
\end{theorem}

\begin{proof}
	By Theorem \ref{Theorem: Univalent Extension}, the univalent extension of a summation is multiplicatively canonical, and so every absolutely $\Sum$-univalent series is contained in the multiplicative fulfillment of $\Sum$. But now $X$ be any series in $\R[[\sigma]]$ that is not absolutely $\Sum$-univalent. We claim $X$ is not in the domain of the multiplicative fulfillment of $\Sum$. Indeed, if $X$ is not absolutely $\Sum$-univalent, then one of the following must occur:
	\begin{caseof}
		\case{$X$ is $\Sum$-infinite over $\Domain$.}{
			In this case, by Theorem \ref{Theorem: Summation-structure of Sum-*** Series}, $X$ is not in the domain of any extension of $\Sum$, and in particular is not in the domain of the multiplicative fulfillment of $\Sum$.
		}
		\case{$X$ is neither $\Sum$-infinite nor $\Sum$-univalent.}{
		In this case, by Theorem \ref{Theorem: Summation-structure of Sum-*** Series}, we may choose extensions $\Sum\prm, \Sum^{\prime \prime} \in \MSums \R \Codomain$ of $\Sum$ with both $\Sum\prm (X)$ and $\Sum^{\prime \prime} (X)$ defined, but with $\Sum\prm (X) \neq \Sum^{\prime \prime} (X)$. Then $X$ is not in the domain of the multiplicative fulfillment of $\Sum$.
		}
		\case{$X$ is $\Sum$-univalent but not absolutely $\Sum$-algebraic.}{
			In this case, we may choose an extension $\Sum\prm \in \MSums \R \Codomain$ of $\Sum$ such that $X$ is $\Sum\prm$-infinite, and as in Case 1 conclude that $X$ is not in the domain of the multiplicative fulfillment of $\Sum$.
		}
	\end{caseof}
	We conclude $\U\Domain$ is the domain of the multiplicative fulfillment of $\Sum$, and so $(\U\Domain, \U\Sum)$ is the multiplicative fulfillment of $(\Domain, \Sum)$.
\end{proof}

We can extend $\U$ to a functor on weakly multiplicative summations by defining $\U \Sum \coloneqq \U \M \Sum$. Naturally, Theorems \ref{Theorem: Univalent Extension} and \ref{Theorem: Univalent Extension is multiplicative fulfillment} still apply in this more general context.

\section{Future work}\label{Section: Future work}

In \cite{DM2} we followed Hardy's philosophy \cite{Hardy} of treating the shift operator $\sigma$ as fundamental to the study of series and summations, and we found the fulfillment of summations taking values in integral domains. In this paper, we shifted our attention to multiplicative summations, and we found the multiplicative fulfillment of summations taking values in an algebraically closed fields. Notably, neither of these are fully general results, so it is natural to try to find fulfillments or multiplicative fulfillments for summations taking values in general commutative rings.

However, there are other lines of inquiry to pursue. Almost all of the dozens of classical summation operators are shift-invariant; so this is an obvious property to axiomatize. If the finitely-supported series are considered as polynomials in $\sigma$, then there is a unique distributive product; we can consider any shift-closed vector space of series as an $R[\sigma]$-module. The extension to the Cauchy product of pairs of series is not automatic, but is natural; with this product, series become formal power series.   
 
Most well-known summations are also invariant under uniform dilutions of the form $\delta_m :\sum_n a_n\sigma^n \mapsto \sum_n a_n \sigma^{mn}$. From an analytic perspective, this is a consequence of the Mellin transform \cite{Muller}. Algebraically, if we want dilutions to have internal multiplicative representations, we are led to the \emph{Dirichlet product}, defined by
\[
\parent{\sum_n A_n \sigma^n} \star \parent{\sum_n B_n \sigma^n} = \sum_n \left (\sum_{k \ell=n} A_k B_\ell \right) \sigma^n
\]
Summations of this type have been studied under such headings as $\zeta$-function regularization \cite{Hawking} and Ramanujan summation \cite{Candelpergher}. We speculate that the formal approach we follow in this paper and in \cite{Dawson, DM2} may be fruitfully applied in this context as well. Indeed, Nori independently developed an analogue to telescopic summation for series indexed by discrete abelian groups \cite{Nori}: more generally, the study of series indexed by groups or even monoids would include both shift-invariation summations and dilution-invariant summations as special cases.
 
In the same way that Euler and Grandi studied specific applications of formal shift-invariance methods without a general theory, Ramanujan \cite{Ram}, \cite{Berndt}[Chapter 6] applied formal methods involving dilution to the series $1+2+3+4+\cdots$, obtaining the value $-1/12$. (Ramanujan did have an analytic theory applying the Euler-Maclaurin summation formula to similar summations \cite{Berndt}, ; see also Hardy \cite{Hardy} [Section 13.10 and 13.17]). Considered as a formal power series, however, that series is $\Add$-infinite, and thus cannot be in the domain of any shift-invariant summation. This contradiction illustrates a tension between dilution-invariance and shift-invariance. It would be nonetheless be interesting to consider fulfillments within families of summations that are both shift-invariant and dilution-invariant.

\bibliographystyle{amsplain}

\end{document}